\documentclass[a4paper]{article}

\usepackage{amssymb,amsmath,url,rotating,ifthen,algorithm,algorithmicx,algpseudocode,epsfig,multirow,array,color,multicol,graphicx,subfig,float,hhline,calc,enumitem}

\usepackage{amsthm}
\newboolean{fabien}
\setboolean{fabien}{false}
\usepackage{booktabs} 
\newcommand{\ra}[1]{\renewcommand{\arraystretch}{#1}}

\usepackage{ifthen}
\usepackage{rotating}

\def\bm{}

\def\E{{\mathbb E}}
\def\NN{\mathcal{N}}
\def\e{\epsilon }
\def\chi{{\mathbf 1}}

\def\PP{{\mathbb P}}

\def\w{\omega }
\def\R{{\mathbb R}}

\def\w{\bm{w}}

\newtheorem{lemma}{Lemma}
\newtheorem{conjecture}{Conjecture}
\newtheorem{proposition}{Proposition}
\newtheorem{property}{Property} 
 
\newtheorem{theorem}{Theorem} 
\newtheorem{remark}{Remark}
\newtheorem{definition}{Definition}

\usepackage{xcolor}
\newcommand{\memo}[2]{\textcolor{#1}{#2}}

\renewcommand{\memo}[2]{} 

\newcommand{\sam}[1]{\memo{cyan}{S.A.M.: #1}}
\newcommand{\mlc}[1]{\memo{magenta}{M.L.C.: #1}}

\newtheorem{innercustomthm}{Theorem}
\newenvironment{customthm}[1]
  {\renewcommand\theinnercustomthm{#1}\innercustomthm}
  {\endinnercustomthm}

\begin{document}


\title{Analysis of Different Types of Regret in Continuous Noisy Optimization}

\author{Sandra Astete-Morales, Marie-Liesse Cauwet, Olivier Teytaud\\
TAO/Inria Saclay-IDF, Univ. Paris-Saclay\\
Bat. 660, rue Noetzlin, Gif-Sur-Yvette, France
}

\maketitle

\begin{abstract}

The performance measure of an algorithm is a crucial part of its analysis.  The performance can be determined by the study on the convergence rate of the algorithm in question. It is necessary to study some (hopefully convergent) sequence that will measure how ``good'' is the approximated optimum compared to the real optimum.

The concept of \emph{Regret} is widely used in the bandit literature for assessing the performance of an algorithm. The same concept is also used in the framework of optimization algorithms, sometimes under other names or without a specific name. And the numerical evaluation of convergence rate of noisy algorithms often involves \emph{approximations} of regrets. We discuss here two types of approximations of Simple Regret used in practice for the evaluation of  algorithms for noisy optimization. We use specific algorithms of different nature and the noisy sphere function to show the following results. The approximation of Simple Regret, termed here Approximate Simple Regret, used in some optimization testbeds, fails to estimate the Simple Regret convergence rate. We also discuss a recent new approximation of Simple Regret, that we term Robust Simple Regret, and show its advantages and disadvantages.

\end{abstract}

\sloppy



\section{Introduction}\label{transfo}
The performance measure of an algorithm involves the evaluation of the quality of the approximated optimum with regards to the real optimum. This can be done by using the concept of \emph{Regret}, well studied in the machine learning literature and used  in the noisy optimization literature, sometimes under other names. Basically, in the optimization framework, the regret acounts for the ``loss'' of choosing the point used in the algorithm over the best possible choice: the optimum. Therefore, we measure the difference between the point used/recommended by the algorithm and the optimum in terms of the objective function.

In general, an optimization algorithm searches for the optimum, and to do so, it produces iteratively \emph{search points} which will be evaluated through the objective function. And at regular steps, the algorithm must return a \emph{recommendation point} that will be the best approximation to the optimum so far. Note that the recommendation point can be equal to a search point, but not necessarily.

The most usual form of regret is termed \emph{Simple Regret}. The \emph{Simple Regret} measures the distance, in terms of fitness values, between the optimum and the recommendation point output by the algorithm. It is widely used (possibly without this name) in noisy optimization \cite{fabian,chen1988,BubeckE09}. However, some test beds, notably the Bbob/Coco framework in the first version, did not allow the distinction between search points (at which the fitness function is evaluated) and recommendations (which are output by the algorithm as an approximation of the optimum), so that the Simple Regret can not be checked. This leads to the use of an \emph{Approximate Simple Regret} (name by us), which evaluates the fitness difference between the search points (and not the recommendations) and the optimum. Later, another form of regret, that we will term here \emph{Robust Simple Regret}, was also proposed, using recommendation points. We analyze in this paper the use of different regrets that aim to estimate the quality of the approximated optimum in a similar way. In particular, we show to which extent they lead to incompatible performance evaluations of the same algorithm over the same class of noisy optimization problems, i.e. the convergence rate for the Approximate or the Robust Simple Regret overestimates or underestimates the convergence rate of the more natural simple regret. We also prove some new results in terms of Simple Regret itself.

\section{Framework and Regrets}\label{sec:framework}

This section is devoted to the formalization of the noisy optimization problem considered and the analysed regrets. We will focus on the \emph{Simple Regret} and the alternative definitions that aim to approximate it (denoted here \emph{Approximate Simple Regret} and \emph{Robust Simple Regret}). At the end of the section we will highlight some general relationships between the presented regrets.

\subsection{Continuous Noisy Optimization}\label{sec:noisyOpt}
Given a fitness function $F:D\subset \R^d\to \R$, also known as objective function, optimization (minimization) is the search for the optimum point $x^*$ such that $\forall x\in D, F(x^*)\leq F(x)$. The fitness function is corrupted by additive noise. In other words, given a search point $x\in D$, evaluating $F$ in $x$ results in an altered fitness value $f(x,\w)$ as follows:   
\begin{equation}\label{additnoise}
	f(x,\w)=F(x)+\w, 
\end{equation}
where $\w$ is an independent random variable of mean zero and variance $\sigma$.  In the present paper, we will consider a simple case, namely a standard Gaussian additive noise 
\footnote{More general cases such that $\E f(x,\w)=F(x)$ can be considered, as most algorithms do not request the noise to be additive and independent; the key point is the absence of bias. \mlc{maybe some cites here? To me, this comment is not clear enough.}\sam{Maybe the comment is not necessary?}}. In addition, 
we assume that $F(x)=\|x-x^*\|^2$, 
where $x^*$ is randomly uniformly drawn in the domain $D$.
\emph{Noisy optimization} is then the search for the optimum $x^*$ such that $\E_\w f(x^*,\w)$ is approximately minimum, where $\E_\w$ denotes the expectation over the noise $\w$.

Consider a noisy black-box optimization scenario: for a point $x$, the only available information is the noisy value of $F$ in $x$ as given by $f(x,\w)$ for some independent $\w$. An optimization algorithm generates $x_1,x_2,\dots,x_n,\dots$, successive \emph{search points} at which the objective function is evaluated in a noisy manner. It also generates
$\tilde x_1,\tilde x_2,\dots,\tilde x_n,\dots$ which are \emph{recommendations} or \emph{approximations of the optimum} after $n$ fitness evaluations are performed. 

\subsection{Simple Regret and variants}\label{sec:criteria}
The \emph{Simple Regret} ($SR$) focuses only on approximating the optimum in terms of fitness values. Its definition is:
\begin{equation*}
 SR_n=\E_\w \left(f(\tilde x_n,\w)-f(x^*,\w)\right) = F(\tilde x_n)-F(x^*)
\end{equation*} Notice that the expectation operates only on $\w$ in $f(\tilde x_n,\w)$, and not on $\tilde x_n$. As a consequence $SR_n$ is a random variable due to the stochasticity of the noisy evaluations of the search points or the (possible) internal randomization of the optimization algorithm. 

The $SR$ can be a part of the performance evaluation of an algorithm. In the noise-free case it can be used to determine the \emph{precision} of a method, by ensuring that the algorithm outputs a recommendation $\tilde{x}_m$ satisfying $SR_m \leq \e$. 
Even more, when testing algorithms, it is common to use the ``first hitting time'' (FHT). FHT in fact refers to the first ``stable'' hitting time, i.e. the next recommendation is at least as good as the previous one. 
This is a reasonable assumption for algorithms solving noise-free problems. In this case the FHT is the minimum $n$ such that $SR_n\leq \e$, provided that the recommendation is defined as $\tilde x_n=x_{i(n)}$ with $1\leq i(n)\leq n$ minimizing $SR_{i(n)}$. However, there is no exact equivalence or natural extension for the FHT with precision $\e$ on noisy optimization. The algorithm only has access to noisy evaluations hence it cannot compute with ``certainty'' $SR_n$, which corresponds to the precision of the algorithm.


An alternative definition, that aims to measure the precision in a similar way to $SR$, is the  \textit{Approximate Simple Regret}\footnote{The name is proposed by us.} ($ASR$), defined by:
\begin{equation*}
 ASR_n=\min_{m\leq n} F(x_m)-F(x^*).
\end{equation*}
$ASR$ takes in account the ``best'' evaluations among all the search points. 
It is used in the Bbob/Coco framework \cite{nbbob1,nbbob2,nbbob3,nbbob4,nbbob5,nbbob6,nbbob7,nbbob8,nbbob9}, and in some theoretical papers \cite{fogaasr}.
Notice that since $ASR$ is non-increasing, the FHT can be computed. 

In this paper we will also discuss another variant of regret, the \emph{Robust Simple Regret}\footnote{Discussed on Bbob-discuss mailing list (\url{http://lists.lri.fr/pipermail/bbob-discuss/2014-October.txt}. The name is proposed by us.} ($RSR$), defined by:
\begin{equation*}
	RSR_n=\min_{k\leq n}\max_{(k-\lfloor g(k)\rfloor)< m\leq k}\left( F(\tilde x_m)-F(x^*)\right),
\end{equation*}
where $g(n)$ is a polylogarithmic function of $n$ and $\lfloor \cdot \rfloor$ is the floor function. 
The $RSR$ is the ``best'' SR since the beginning of the run, sustained during $\lfloor g(k)\rfloor$ consecutive recommendations \footnote{If $(k-\lfloor g(k)\rfloor)<0$, then the $\max$ on the definition considers indexes between $1$ and $k$}. The polylogarithmic nature of $g(\cdot)$ is explained by the following argument: $g(k)$ be large enough, so that we have a 
correct recommendation confirmed over $g(k)$ iterations, but small enough, so that we do not have to wait many evaluations  before acknowledging that a correct recommendation has been found.
The $RSR$ uses the recommendation $\tilde x_m$ instead of the search point $x_m$ used in $ASR$. But it uses the worst of a sequence of recommendations. 

As a side note about the definition of $RSR$, it was originally proposed to use a quantile instead of the maximum. The ``quantile version'' (without this name), was proposed to become part of the performance measure in Bbob/Coco. However, we will show that it is possible to get a $RSR$ decreasing quicker than the $SR$, so that $RSR$ is a poor approximation of $SR$. The result is valid even with the quantile $100\%$, i.e. the maximum. The same is possible with any other quantile. 

The introduction of $RSR$ apparently outplays $ASR$ as an approximation of $SR$ by two means. First, by using recommendations rather than search points. Second, by checking on multiple recommendations that the optimum is correctly found with a given precision. In addition, as well as $ASR$, it is non-increasing, therefore it can be used for fastening experiments on testbed. Please note however that this advantage makes sense only when the target fitness value is known, which is rarely the case except in an artificial testbed.

To investigate the convergence rate of the regrets, we will use a slightly different notation than classical works on noisy optimization. Usually the rates are given in terms of $O(h(n))$ where $h(n)$ is some function depending on the number of evaluations $n$. The state of the art shows that in many cases (\cite{dupac,fabian,shamir,fogaasr}) $SR_n = O(n^\psi)$ where $\psi<0$ implies that the algorithm converges. Therefore, there is a linear relationship between $\log(SR_n)$ and $\log(n)$ with a \emph{slope} $\psi$, where $\log(\cdot)$ is the natural logarithm. We will then refer to the \emph{slope of the regret} when speaking about the convergence rate of the regret. The definition of the \emph{slope of the SR} is:

\begin{equation*}
s(SR) = \underset{n\to\infty}{\lim\sup}~ \frac{\log(SR_n)}{\log(n)}
\end{equation*}

We have the corresponding definition for the slope of $ASR$ and $RSR$. Notice that if the slope is close to $0$, then the algorithm  (at best) converges \emph{slowly}. On the contrary, if the slope is negative and far away from $0$, then the algorithm is \emph{fast}.

\subsection{General results for $SR$, $ASR$ and $RSR$: $RSR$ is an optimistic approximation of $SR$}

The problems analysed in this paper arise from the gap between $s(SR)$ and $s(ASR)$ or $s(RSR)$. Ideally we would like to have a regret that can be used easily and that \emph{approximates} the Simple Regret. In the following sections (\ref{sec:EA} and \ref{sec:stoc}) we will see with specific examples there is indeed a gap between $s(SR)$ and $s(ASR)$. In some cases using $ASR$ overestimates the performance of algorithms and in others it underestimates their performance. An extreme case is detailed in section \ref{sec:stoc}, where we prove that Alg.~\ref{alg:shamir} has optimal convergence rate in term of $SR$ whilst for $ASR$ it does not converge at all.

In general, by definition, we have that for any algorithm,
$s(RSR)\leq s(SR)$. From this point of view, $RSR$ is a correct lower bound for $SR$. In other words, if an algorithm is fast in terms of $SR$, its performance measured by $RSR$ will be at least as good. 
Unfortunately, this bound is not nearly tight. We will prove that a small modification on the algorithm induces $s(RSR)\leq s(ASR)$ whereas $s(SR)$ is the same - so that, for cases in which $s(ASR)<s(SR)$ (sometimes by far, as explained in later sections), we get $s(RSR) < s(SR)$ (sometimes by far).


Let $A$ be an optimization algorithm and its search points $(x_i)_{i \geq 1}$. Consider another algorithm denoted $A_g$ and its successive search points $(X_i)_{i \geq 1}$. The search points of $A_g$ are obtained by repeating $\lfloor g(n)\rfloor$ times, for any $n$, the search points $x_n$ of $A$. Hence, we get the assignment:  

\begin{tabular}{rcrrl} 
$X_{1}=$ & $\dots$ & $=$ & $ X_{1+\lfloor g(1)\rfloor}$ & $\leftarrow x_1$\\ 
$X_{2+\lfloor g(1) \rfloor}=$ & $\dots$& $=$ & $X_{2+\lfloor g(1)\rfloor+\lfloor g(2)\rfloor} $ & $\leftarrow x_2$\\
& $\vdots$ & \\
$X_{n+\sum_{j=1}^{n-1} \lfloor g(j)\rfloor} =$& $\dots$ & $=$ & $X_{n+\sum_{j=1}^{n} \lfloor g(j)\rfloor} $& $\leftarrow x_{n}$
\end{tabular} 

Let the recommendation points of $A_g$ be defined by $\tilde{X}_n = X_n$ for any $n$.  Therefore $A_g$ is a slightly modified version of $A$ since there is an additional polylogarithmic number of evaluation in $A_g$, assuming $g$ is polylogarithmic. The $RSR$ of $A_g$ (say $RSR_{A_g}$) converges approximately as fast as  the $ASR$ of the original algorithm (say $ASR_A$). The extra polylogarithmic number of evaluations  does not affect the linear convergence in log/log scale. Hence, for any algorithm $A$, $s(RSR_{A_g}) \leq s(ASR_A)$. 

The general relationships between the slopes of the $SR$ and its approximations are not conclusive since the bounds are not tight. We will show some gaps between the different approximations of $SR$ and the $SR$ itself.
 In the following sections we present five algorithms that will serve as clear examples to see the differences of using one or another regret as performance measures. 
We will focus in two types of algorithms: the first group, in Section~\ref{sec:EA}, consists of Evolutionary Algorithms and Random Search and the second, in Section~\ref{sec:stoc}, of algorithms using approximations of the gradient of the objective function.
For each class of algorithms, we exhibit convergence rate bounds on $s(SR)$, $s(ASR)$ and $s(RSR)$. Section~\ref{sec:experiments} displays some experimental works in order to confront theory, conjecture and practice.

\section{Evolutionary Algorithms}\label{sec:EA}
On the group of Evolutionary Algorithms (EAs), we present Random Search (RS), Evolution Strategy (ES) and Evolution Strategy with resampling (ES+r). They all use comparisons between fitness values to optimize the function. 

\vfill\break
\subsection{Random Search}
Random Search (Alg. \ref{alg:rs}) is the most basic of stochastic algorithms \cite{rastrigin1964convergence}. 
The search points $x_1,\dots,x_n,\dots$ are independently identically drawn according to some probability distribution. $\tilde x_n$ is usually the search point with the best fitness so far, i.e. with $y_i$ the fitness value obtained for $x_i$, we have $\tilde x_n=x_i$ with $i\in\{1,\dots,n\}$ minimizing $y_i$. 

\begin{algorithm}[H]
\caption{\label{alg:rs} {\scriptsize Random Search.}}
\begin{algorithmic}[1]\scriptsize
	\State{{\bf Initialize:} Candidate solution $\tilde x$ randomly drawn in $[0,1]^d$}
	\State{$bestfitness \leftarrow f(\tilde x)$}
	\State{{\bf Initialize:} $n\leftarrow1$}
	\While{not terminated}
		\State{Randomly draw $y$ in $[0,1]^d$.}
		\State{$fitness\leftarrow f(y)$}
		\If{$fitness<bestfitness$ }\label{line:sel1}
			\State{ $\tilde x \leftarrow y$} 
			\State{ $bestfitness\leftarrow fitness$}\label{line:sel2}
		\EndIf
		\State{$n\leftarrow n+1$}
	\EndWhile 
	
	\Return{$\tilde{x}$}
\end{algorithmic}
\end{algorithm}

We consider in this paper a simple variant of RS to show clearly the contrast between $SR$ and $ASR$.

{\bf Framework for RS:} each search point is randomly drawn independently and uniformly, sampled once and only once, with the uniform probability distribution over $[0,1]^d$. The objective function is the noisy sphere function $f$:
\begin{equation}
f(x)=\|x-x^*\|^2+\NN.\label{noisysphere}
\end{equation}
where $\NN$ is a standard Gaussian variable.
 
In this setting, existing results in the literature imply a bound on $s(ASR)$ as explained in Property \ref{prop:RS}. We will prove then that the slope of the Simple Regret is not negative, as formalized in Theorem \ref{rswn}.

\subsubsection{Approximate Simple Regret: $s(ASR)=-2/d$}

\begin{property}\label{prop:RS}
Consider RS described in Alg.~\ref{alg:rs}, with the framework above.
Then almost surely $ASR_n=O\left(\frac{1}{n^{2/d}}\right)$. 
\end{property}

\begin{proof}
 \cite{deheuvels} has shown that among $n$ points generated independently and uniformly over $[0,1]^{d}$ the closest  
  search point to the optimum  is almost surely at distance $O\left(\frac{1}{n^{1/d}}\right)$ from the optimal point $x^*$ within a logarithmic factor. Hence for the sphere function\footnote{The result also holds for a function locally quadratic around a unique global optimum.}, the Approximate Simple Regret $ASR_n$ almost surely satisfies: $ASR_n = O\left(\frac{1}{n^{2/d}}\right)$ up to logarithmic factors. 
\end{proof}


\subsubsection{Simple Regret: $s(SR)$ is not negative}


%

\begin{theorem}\label{rswn}
With the framework above, for all $\beta>0$, the expected simple regret $\E(SR_n)$ is not $O(n^{-\beta})$.
\end{theorem}
\begin{proof} See supplementary material. \end{proof}

{\bf Remark:} Roughly speaking, the proof of the theorem is based on the fact that with a non-zero probability a search point which does not have the best fitness value, is selected as the best point in Lines~\ref{line:sel1}-\ref{line:sel2} of Alg. \ref{alg:rs}.

\def\movedintosup{
The following sections prove this theorem.\mlc{if we are too long, we can put this proof in supplementary material.}

\begin{lemma}[the quantiles of the standard
Gaussian random variable.]

Let $Q(q)$ be the quantile of the standard centered Gaussian $\NN$,
i.e. $\forall q\in (0,1), P(\NN \leq Q(q) ) =q$.
       
Then, there exist $X(q)=-\sqrt{-\log(q S_X(q) )}$ and $Y(q)=-\sqrt{-\log(q S_Y(q))}$ such that
$$X(q) \leq Q(q) \leq Y(q),$$
for some $S_X(q)$ and $S_Y(q)$ polylogarithmic as functions of $q$.
\end{lemma}
Proof: See e.g. http://www.johndcook.com/normalbounds.pdf TODO we should put this into a clean reference.QED
Let us consider $N$ the number of search points,
$x_1,\dots,x_N$ the $N$ i.i.d search points,
$\epsilon(N)>0$ and $C(N)>0$
with $\epsilon(N) =o( C(N) )$,
and $\epsilon(N)$ and $C(N)$ both $\Omega(1/N)$ and $o(1)$.

Define $N_g(N)=\{ i\in \{1,\dots,N\}; || x_i - x* || \leq \sqrt{\epsilon(N)} \} $ the number of search points with norm $\leq \sqrt{\epsilon(N)}$ (i.e. ``good'' search points, with simple regret better than $\epsilon(N)$ ).

Define  $N_b(N)$ the number of search points with norm $> \sqrt{\epsilon(N)}$ but norm $\leq \sqrt{C(N)}$ (i.e. bad search points, but not very bad...), i.e. $N_b(N)\leq N-N_g(N)$. \sam{precise here if we abuse of notation or change the notation ($N_{b}(N)$ is a set) }
      
\begin{lemma}[Linear numbers of good and bad points]\label{lemmaCheb}
There exist constants $K_g>0$ and $K_b>0$, such that with probability $\geq \frac12$, 
     \begin{equation}
N_g(N) \leq K_g N\sqrt{\epsilon(N)}\label{calvin}
\end{equation}
and 
\begin{equation}
N_b(N) \geq K_b N \sqrt{C(N)}.\label{hobbes}
\end{equation}
   \end{lemma}
\begin{proof} 
Consider a search point $x_i$. It is ``good'', in the sense above, with probability $\Omega(\sqrt{\epsilon_N})$. This holds for each of the search points; therefore the number of good points is the sum of $N$ Bernoulli random variables with parameter $\Omega(\sqrt{\epsilon_N})$. The expectation and the variance are therefore $\Omega(N\sqrt{\epsilon_N})$.
By Chebyshev inequality, there is $\alpha>0$ such that a random variable $X$ is $O(\E X+\alpha \sqrt{Var\ X})$ with probability at least $\frac12$. This implies that the number $N_g$ of ``good'' points is 
\begin{equation}
N_g=O(N\sqrt{\e_N}+\alpha\sqrt{N\sqrt{\e_N}}).\label{bolalapouf}
\end{equation}
By the assumption $\e_N=\Omega(1/N)$, $N\sqrt{\e_N}\to\infty$; therefore
Eq. \ref{bolalapouf} implies $N_g=O(N\sqrt{\e_N})$.
The proof is similar for the number of ``bad'' points.
\end{proof}

Consider $G=\{i\in \{1,\dots,N\}; ||x_i-x^*||\leq \sqrt{\epsilon(N)}\}$ the indices of good search points, and $B=\{i\in\{1,\dots,N\}; \sqrt{C(N)} \geq ||x_i-x^*||>\sqrt{\epsilon(N)}\}$ the indices of bad search points. 

\begin{proposition}\label{zeprop}
For some $c>0$, for all $N$, with probability at least $c$,
the minimum of the noisy fitness values for the $N_g$ good points
verifies $\inf_{i \in G} y_i \geq X(1/N_g)$, and the best noisy fitness for the $N_b$
bad points verifies $\inf_{i\in B} y_i \leq C+Y(1/N_b)$.
\end{proposition}

\begin{proof}

   Consider the case in which $N_g$
   and $N_b$ are upper and lower bounded (respectively) as explained
   in Lemma \ref{lemmaCheb}.

  This happens with probability $\geq \frac12$, by that lemma.

 Consider some $n_i$ (for $i\in \{1,2,\dots,N\}$), which are independently identically distributed according to some absolutely continuous density. 
Define $q_N$ the $\frac1N$ quantile of their common probability distribution.
Then the probability that $\inf \{n_1,n_2,\dots,n_N\}$ is less or equal to $q_N$ is, by definition, 
$1-(1-1/N)^N$. This converges to the constant $1-\exp(-1)$ as $N\to \infty$.

   Therefore, the probability that one of the $N_g$
   good points has a noise $\leq Q(1/N_g)$ is
   upper and lower bounded by a constant;
   the probability that one of the $N_b$ bad
   points
   points has a noise $\leq Q(1/N_b)$ is
   upper and lower bounded by a constant;
   these events are independent, so the
   probability that both happen simultaneously
   is a constant; these events are also
   independent of the 0.5 probability from Lemma \ref{lemmaCheb},
   so with lower bounded probability all
   these events happen simultaneously.\end{proof}


{\bf{Proof of Theorem \ref{rswn}:}}
\begin{proof}
Let us assume that the expected
  simple regret has slope $<-\beta$ for some
$1>\beta>0$; then define
 $\epsilon(N)=N^{-\beta}$
and $\alpha=\beta/k$
for some $0<k<1$, and $C(N)=N^{-\alpha}$;
then Lemma \ref{lemmaCheb} and Proposition \ref{zeprop}  implies that there is a $c>0$ such that with probability $c$, for $N$ sufficiently large,
\begin{itemize}
	\item there are much more good points than bad points, i.e.
\begin{equation}
	N_b=o(N_g)\label{proutproutprout}
\end{equation}
 	thanks to Eq. \ref{calvin} and \ref{hobbes}.
	\item all good points have noisy fitness at least $X(1/N_g)=\tilde \Theta(-\sqrt{-\log(N_g)})$;
	\item at least one bad point has fitness at most $C(N)+Y(1/N_b)=\tilde\Theta(N^{-\alpha}-\sqrt{-\log(N_b)})=\tilde\Theta(-\sqrt{-\log(N_b)})$; 
	\item therefore (by the two points above, and using Eq. \ref{proutproutprout}, one bad point is selected; \sam{when is it selected? we always keep the point with the best fitness}
\end{itemize}
so that the simple regret is larger than $\epsilon(N)$
with probability $c$ for all $N$ sufficiently
large. 

This implies that the expected simple
regret is at least $c \epsilon(N)$, which contradicts the slope $<-\beta$. \end{proof}
}

\subsection{Evolution Strategies ($ES$)}
Evolution Strategies~\cite{Rechenberg73,Schw74b} are algorithms included in the category of Evolutionary Algorithms (EAs). 
In general, EAs evolve a population until they find an optimum for the objective or fitness function. The process starts by a population randomly generated. Then the algorithm iterates creating new individuals using crossover and mutation and then evaluating this new population of offspring and selecting the ones - regarding to their fitness values - that will become the parents of the next generation. 

ES have some more specific selection and mutation processes. Usually the mutation is performed by creating new individuals starting from the parent and adding a random value to it (usually normally distributed around the parent). There are various rules for choosing the step-size.
The selection in ES is usually deterministic and rank-based. This is, the individuals  chosen to be the parents of the next generations are the ones that have the best fitness values. 

When dealing with noisy function, the sorting step of the ES is disturbed by the noise and misranking might occur. To tackle this problem, Arnold and Beyer, in \cite{abinvestigation,abnoise} propose to increase the population size. An alternative is to evaluate multiple times the same search point and average the resulting fitness values. For a given search point $x\in D$, $r$ evaluations are performed: $\left(f(x,w_1), \dots, f(x,w_r)\right)$ and the fitness value used in the comparisons is the average of these evaluations $\frac{1}{r} \sum_{i=1}^{r} f(x,w_i)$. In particular, the variance of the noise is divided by $r$. Several rules have been studied: constant \cite{bignoise3}, adaptive (polynomial in the inverse of the step-size), polynomial and exponential \cite{bignoise2} number of resamplings. A general $(\mu,\lambda)$-ES is presented in Algorithm \ref{alg:es}. 

\begin{algorithm}[H]
\caption{\label{alg:es} {\scriptsize $(\mu,\lambda)$-Evolution Strategy. The resampling function $r$  may be constant or depend on the number of iterations and possibly on the step-size. When $r=1$, the algorithm reduces to an ES without resampling. ${\cal{N}}$ is a standard Gaussian of dimension $d$.  Here, the index $n$ is the number of \emph{iterations}}}
\begin{algorithmic}[1]\scriptsize
	\State{{\bf Input:} Parameters $\mu$, $\lambda$ and resampling function $r$ }
	\State{{\bf Initialize:} Parent $\tilde x$ and Step-size $\sigma$}
	\State{{\bf Initialize:} $n\leftarrow1$}
	\While{not terminated}
		\State{{\bf Mutation step:}} $\forall i\in \{1,\dots,\lambda\}$,
		$x^{(i)}\leftarrow \tilde{x}+ \sigma\NN$\label{line:mutate}
		\State{{\bf Evaluation step:}}
		
		$\forall i\in \{1,\dots,\lambda\}$
${y^{(i)} \leftarrow\frac{1}{r(n)}\sum_{j=1}^{r(n)}f(x^{(i)}, w_{j})}$\label{line:eval}
		\State{{\bf Selection step:} Sort the population according to their fitness and select the $\mu$ best:  $(x^{(i)})^\mu$}
		\State{{\bf Update Parent:} $\tilde{x}$ from $\sigma$, $(y^{(i)})^\mu$ and $(x^{(i)})^\mu$}
		\State{{\bf Update Step-size:} $\sigma$ from $\sigma$, $(y^{(i)})^\mu$ and $(x^{(i)})^\mu$}
		\State{$n\leftarrow n+1$}
	\EndWhile
	
	\Return{$\tilde{x}$}
\end{algorithmic}
\end{algorithm}



\subsubsection{Regrets for ES without resampling}\label{subsubsec:ES}

It is known \cite{stocopti5} that when the noise strength is too big, classical evolution strategies (without reevaluations or other noise adaptation scheme) do not converge, they stagnate. \cite{BeyerMutate} experimentally shows that an ES without any adaptation to the noisy setting stagnates around some step-size and at some distance of the optimum. The divergence results suggest that the ES in this case is only as a more sophisticated version of RS.  The steps of the ES are more complicated, but not sufficiently adapted to handle the noise of the function. We propose then a Conjecture on the convergence rates for ES.  

\begin{conjecture}[Convergence rates for ES]\label{conj:es}
Evolution Strategies without a noise handling procedure have the same convergence rates as Random Search for all regrets.
\end{conjecture}
 
\subsubsection{Simple regret for ES with resamplings}\label{SR_ES}

We will see  that the results are more encouraging than in Section \ref{subsubsec:ES} when we consider an ES with some adaptation to mitigate the effect of the noise.  We will assume that the function $r$ (number of revaluations per point)  in Alg.~\ref{alg:es} grows polynomially or exponentially with the number of iterations.

The work in \cite{bignoise2} shows that ES that include an exponential number of revaluations converges with high probability to the optimum. The convergence rate is $s(SR) = K$ for some $K<0$ under assumptions about the convergence in ES in the noise-free case. Moreover  \cite{esareslow} shows that ES, under general conditions, must exhibit $K> -\frac12$.
There is no formal proof of an upper bound that can theoretically ensure a value or a range for $s(SR)$. 
However, the experiments on \cite{bignoise2} suggest that the $K=-\frac12$ is reached for functions with a quadratic Taylor expansion and additive noise (as in Eq.~\ref{additnoise}).
Hence we propose Conjecture~\ref{conj:esSReval}:

\begin{conjecture}[$SR$ for ES + $r$]\label{conj:esSReval}
Consider $0<\delta<1$. For some resampling parameters {{(i.e. for some revaluation function $r$)}}, Evolution Strategies with resampling (Alg.~\ref{alg:es}) satisfy $s(SR)=-1/2$ with probability $1-\delta$.
\end{conjecture}
This conjecture applies to some ES with step-size scaling as the distance to the optimum, i.e. $\sigma_n$ used for generating the $n^{th}$ search point has the same magnitude as $\|\tilde{x}_n - x^{*}\|$ (\cite{Rechenberg73,Beyer:bookES}). \cite{BeyerMutate} has proposed variants of ES for quickening the convergence thanks to large mutations and small inheritance. Such an approach is not covered by the bound in \cite{esareslow} and it is for sure an interesting research direction - maybe it might reach slope $s(SR)=-1$.

\subsubsection{Approximate simple regret for ES with resamplings}\label{asrreval}

We have seen that an ES can reach a slope of $SR$ approximately $-\frac12$, when using resamplings.
However, $ASR$ can be better 
by slightly modifying the original ES, and therefore achieving a faster convergence rate than the real one represented by $s(SR)$
We consider an ES - called $MES+R$ for Modified ES with Resampling. Let $r_n$ be exponential in the number of iterations: $r_n=R\cdot \zeta^n$, $R>0$, $\zeta>1$. 
$MES+R$ is as in Alg. \ref{alg:es} with the following modifications. At iteration $n$:

{\bf Generation:} (Alg.~\ref{alg:es}, Line \ref{line:mutate}) 
Generate additional $r_n$ ``fake'' offspring $\{ x^{(i)f}: 1\leq i\leq r_n \}$, with the same probability distribution as the $\lambda$ offspring. They will be evaluated one time each, but they will {\emph{not}} be taken into account for the selection. Note that this means that they are part of the sequence of points considered by $ASR$, but not by $SR$.

{\bf Evaluation:} (Alg. \ref{alg:es}, Line \ref{line:eval}) Evaluate $r_n$ times each ``true'' offspring $\{ x^{(i)} : 1\leq i\leq\lambda \}$ to obtain their corresponding fitness value $y^{(i)}$. Evaluate one time each ``fake'' offspring. 
Therefore, performing $(\lambda+1)r_n$ function evaluations in each iteration.

Then, under some reasonable convergence assumptions which are detailed in theorem \ref{thm:asr} below, the $ASR$ reaches a faster rate: $s(ASR)=-1/2-2/d$ with high probability.


\begin{theorem}\label{thm:asr}
Consider $0<\delta<1$.
Consider an objective function $F(x)=\|x\|^2$, where $x\in \R^d$.
Consider a $MES+R$ as described previously. Assume that:
\begin{enumerate}[label=(\roman*)]
\item $\sigma_n$ and $\|\tilde{x}_n\|$ have the same order of magnitude: 
\begin{equation}\label{xsigma}
\|\tilde{x}_n\|=\Theta(\sigma_n).
\end{equation}
\item\label{logalso} $\log-\log$ convergence occurs for the $SR$: 
\begin{equation}\label{loglog}
\frac{\log(\|\tilde{x}_n\|)}{\log(n)}\underset{n\rightarrow+\infty}{\longrightarrow }-\frac12 ~\mbox{ with probability $1-\delta$,}
\end{equation}
\end{enumerate}

Then, with probability at least  $1-\delta$, $s(ASR)=-1/2-2/d$.
\end{theorem}
\begin{proof} See supplementary material. \end{proof}

\begin{remark}\label{rmq:sr}
The assumption of $s(SR)=-1/2$ is based on the convergence of ES in the noise-free case and it is essential to prove Theorem~\ref{thm:asr}. This rate of convergence can be proved when the ES converges in the noise-free case (details on  \cite{bignoise2}).
But the convergence of ES has not been formally proved, not even for the noise-free case. There is an important element given in \cite{TCSAnne04-corr}, showing that $\frac1n \log ||x_n-x^*||$ converges to some constant, but this constant is not proved negative. Furthermore, parameters ensuring convergence in the noisy case are unspecified in \cite{bignoise2}. 
\end{remark}

\section{Stochastic Gradient Descent}\label{sec:stoc}
For the group of Stochastic Gradient Descent Algorithms, we consider the ones presented by Shamir \cite{shamir} and Fabian \cite{fabian} which approximate the gradient of the objective function. We will denote them Shamir algorithm and Fabian algorithm respectively. Both of them approximate the gradient of the function using function evaluations by different methods, therefore they remain in the black-box category. Fabian algorithm uses the average of redundant finite differences and Shamir algorithm a one point estimate gradient technique. 

The convergence rates in terms of $SR$ are proved in~\cite{shamir} and ~\cite{fabian}. For Shamir it is shown that $s(SR)=-1$ in expectation for quadratic functions.
Fabian ensures a rate $s(SR)=-1$ approximately and asymptotically only for limit values of parameters. However, it requires only smooth enough functions, so the class of functions is wider than the one considered in~ \cite{shamir}. This rate $s(SR)=-1$ has been proved tight in ~\cite{chen1988}. Hence, Shamir and Fabian  algorithm are  faster than ES's, which cannot do better than $s(SR)=-\frac12$, at least under their usual form~\cite{esareslow}. 

\subsection{Shamir's quadratic algorithm}

Shamir algorithm presented in~\cite{shamir} for quadratic functions  is Algorithm \ref{alg:shamir}.
\begin{algorithm}
	\caption{ {\scriptsize \label{alg:shamir} Shamir Algorithm for Quadratic functions. $\Pi_{W} $ represents the projection over the space $W$} }
\begin{algorithmic}[1]\scriptsize
	\State{{\bf Input:} Parameters $\lambda$ and $\e$}
	\State{{\bf Initialization:} $\hat{x}_1\leftarrow 0$, $n\leftarrow 1$}
	\While{not terminated}
		\State{Pick $r\in \{-1,1\}^{d}$ uniformly at random}
		\State{$x_{n}\leftarrow\hat{x}_{n}+\frac{\e}{\sqrt{d}}r$}
		\State{{\bf Evaluate:} $v\leftarrow f(x_n,w)$}
		\State{$\hat{g}\leftarrow\frac{\sqrt{d}v}{\e}r$}
		\State{$\hat{x}_{n+1}\leftarrow\Pi_{W}\left(\hat{x}_{n}+\frac{1}{\lambda n}\hat{g} \right)$}
		\State{{\bf Recommend:} $\tilde{x}_n\leftarrow\lceil\frac{2}{n}\rceil\sum_{k=\lceil n/2\rceil}^{n} \hat{x}_{k}$}
		\State{$n\leftarrow n+1$}
	\EndWhile

	\Return{$\tilde{x}_{n}$}
\end{algorithmic}
\end{algorithm}

One of the key points in Alg.~\ref{alg:shamir} are that there is only one evaluation per iteration (somehow in the spirit of Simultaneous Perturbation Stochastic Approximation SPSA~\cite{spall00adaptive,beynoise}). The second important point is that the expectation of the distance between search points and recommendations is constant, which implies that the search points do not converge towards the optimum! This is not a problem for the convergence in terms of $SR$, when search points $x_n$ and recommendations $\tilde{x}_n$ are distinguished, but it makes a difference for  $ASR$. 

Shamir algorithm has an optimal convergence rate in expectation ($s(SR)=-1$) for quadratic functions. This fact should be acknowledge by any other regret  used to evaluate its performance which aims to aproximate the $SR$.
But intuitively in the framework of Shamir algorithm, the $s(ASR)$ is presumably a bad approximation of $s(SR)$  
due to the queries at a constant distance of the current recommendation. This convergence rate in terms of   $s(ASR)$ could not be obtained from the results in ~\cite{shamir}. Nonetheless, we prove in a general way that as long as the results for Shamir algorithm are satisfied \emph{almost surely}, then $s(ASR)=0$ a.s. Therefore we present the latter result in  Theorem ~\ref{thm:aSRgradient} and a conjecture on the convergence rate of $s(ASR)$ in expectation for Shamir algorithm.

\begin{theorem}
\label{thm:aSRgradient}
Assume that the optimum $x^*$ is unique and that $(\tilde x_{n})$ is a sequence of recommendation points converging a.s. to $x^*$. Assume that the sequence of evaluation points $(x_n)$ is such that $\forall n, x_n\neq x^*$ a.s. and that $ \|x_n-\tilde{x}_n\| $ is constant.
Then, a.s.
\begin{equation*}
s(ASR)=0.
\end{equation*}
\end{theorem}

\begin{proof}
$\tilde x_n$ converges almost surely to the optimum and $x_n$ is at a constant distance from $\tilde x_n$.
Therefore the distance between $x_n$ and the optimum converges to a constant.
This implies that the minimum $\min_{i=1}^n \|x_i-x^*\|^2$ is lower bounded by some constant.
Therefore $s(ASR)=0$.
\end{proof}


\begin{conjecture}[$ASR$ for the Shamir algorithm]
	 Shamir algorithm also verifies $s(ASR)=0$ a.s. on quadratic functions.
\end{conjecture}

\subsection{Fabian Algorithm}


Algorithm \ref{alg:fabian} presents the algorithm studied in \cite{fabian}. Unlike Algorithm~\ref{alg:shamir}, Fabian algorithm performs several evaluations per iteration, and the distance between search point and recommandation decreases. 

\begin{algorithm}
\caption{ {\scriptsize \label{alg:fabian} Fabian Algorithm. $e_i$ is the $i^{th}$ vector of the standard orthogonal basis of $\R^d$ and $e_{1,s/2}$  is the $1^{st}$ vector of the standard orthogonal basis of $\R^{s/2}$. $v_i$ is the $i^{th}$ coordinate of vector $v$. ($\hat{x}_{i}$) is the $i^{th}$ coordinate of intermediate points ($\hat{x}$). ($x^{(i,j)+}$) and ($x^{(i,j)-}$) are the search points and $\tilde{x}$ is the recommendation. Here, the index $n$ is the number of \emph{iterations}.} }
\begin{algorithmic}[1]\scriptsize
	\State{{\bf{ Input:}} An even integer $s>0$. Parameters $a$, $\alpha$, $c$, $\gamma$.}
	\State{{\bf Initialization:}\\
	 $u_{i} \leftarrow \frac{1}{i},\ \forall\ i\in\{1,\dots,s/2\}$\\
	 Matrix $U\leftarrow\left( \| u_{j}^{2i-1} \| \right)_{1\leq i,j \leq s/2}$\\
	 Vector $v\leftarrow\frac12 U^{-1}e_{1,s/2}$\\
	$\tilde{x}\leftarrow x\in [0,1]^{d}$ uniformly at random\\
	$n \leftarrow 1$}
	\While{not terminated}
		\State{$a_n \leftarrow\frac{a}{n^{\alpha}}$, $c_n\leftarrow\frac{c}{n^{\gamma}}$}
		\State{$\forall j\in\{1,\dots,s/2\}$,\ $\forall i\in\{1,\dots,d\}$\\
		{~~~~\bf Evaluate:} 
		\begin{eqnarray*}
		x^{(i,j)+}\leftarrow \tilde{x} + c_n u_{j}e_{i} \qquad
		x^{(i,j)-}\leftarrow \tilde{x} - c_n u_{j}e_{i}
		\end{eqnarray*}}
		\State{$\hat{x}_{i}\leftarrow\frac{1}{c_n} \sum_{j=1}^{s/2} v_j \left(f(x^{(i,j)+})-f(x^{(i,j)-})\right)$}
		\State{{\bf Recommend:} $\tilde{x}\leftarrow\tilde{x}-a_{n}\hat{x}$}
		\State{$n\leftarrow n+1$}
	\EndWhile

	\Return{$\tilde{x}$}
\end{algorithmic}
\end{algorithm}

The work in  \cite{fabian} gives the convergence rate in terms of $SR$. The result is presented here as Theorem \ref{thm:fabSR}. The value of the $s(SR)$ depends on the parameters of the algorithm and it is ensured a.s.

\begin{theorem}[Simple Regret of Fabian algorithm]\label{thm:fabSR}
Let $s$ be an even positive integer and $F$ be a function {${(s+1)}$-times differentiable} in the neighborhood of its optimum $x^*$. Assume that its Hessian and its $(s+1)^{th}$ derivative are bounded in norm. Assume that the parameters given in input of Algorithm \ref{alg:fabian} satisfy: $a>0$, $c>0$, $\alpha= 1$, $0<\gamma<1/2$ and $2\lambda_0 a >\beta_0$ where $\lambda_0$ is the smallest eigenvalue of the Hessian. Let ${\beta_0=\min\left(2 s \gamma, 1 - 2 \gamma \right)}$. Then, a.s.:
\begin{equation}\label{fabiantropfort}
n^{\beta}(\tilde{x}_n-x^*) \rightarrow 0\ \forall\ \beta <\beta_0/2
\end{equation}
In particular, when $F$ is smooth enough, we get ${s(SR)=-2\beta}$.
\end{theorem}

\begin{remark}
Note that $s(SR)$ optimal when $\gamma = \frac{1}{2}(s+1)^{-1}$. In this case, $\beta_0=\frac{s}{s+1} \underset{s\rightarrow \infty}{\rightarrow} 1$. $\beta_0$ can be made arbitrarily close to $1$, so $2\beta$ also, but then $\gamma$ goes to $0$. Hence we get the values of Table \ref{maintable}, with $2\beta=1-e$, $e>0$ and close to $0$.
\end{remark}
This shows that the Fabian algorithm can have $s(SR)$ arbitrarily close to $-1$, which is optimal. As in the case of Shamir, this optimal performance should be captured by the regret used to evaluate the algorithm. 
Unfortunately, this is not the case, as detailed in Theorem \ref{thm:fabASR}.


\begin{theorem}[$ASR$ of Fabian algorithm]\label{thm:fabASR}
Let $F$ be a $\lambda$-convex and $\mu$-smooth function corrupted by an additive noise with upper bounded density and with optimum randomly drawn according to a distribution with upper bounded density. Then, a.s.,
\begin{equation*}
s(ASR)= -\min(2\beta,2\gamma).
\end{equation*}
\end{theorem}

\begin{proof} See supplementary material. \end{proof}

\begin{remark} 
Theorem \ref{thm:fabASR} shows that $s(ASR)=-\min(2\beta,2\gamma)$, i.e., when then Fabian algorithm is optimized for $SR$, $s(ASR)$ is close to $-2\gamma$, close to $0$.
\end{remark}

\subsection{Shamir and Fabian adapted for $ASR$}
Both algorithms presented in this section have a clear difference between the search and recommendation points. This fact is not automatically distinguished when we are evaluating their performance using for example a test bed. If we modify the algorithms we can achieve $ASR$ approximating well the optimal behavior reported by $SR$. A modification such as sampling  one point out of two at the current recommendation, without using it in the algorithm can imply $s(ASR)=s(SR)$ arbitrarily close to $-1$.
\begin{sidewaystable}
\centering
\ra{1.3}
\begin{tabular}{@{}rrrcrrcrr@{}}\toprule
& \multicolumn{2}{c}{SR} & \phantom{abc}& \multicolumn{2}{c}{ASR} &
\phantom{abc} & \multicolumn{2}{c}{RSR}\\ \cmidrule{2-3} \cmidrule{5-6} \cmidrule{8-9}
& conv. & type && conv. & type && conv. & type  \\ \midrule
\bf{Evolutionary Algorithms}\\
RS & $\mathbf{0}$  & \textbf{expect.} && $\mathbf{ -\frac2d }$ & \textbf{a.s.} && $\mathbf{ -\frac2d }$ & a.s. \\
ES &  $0$ & expect. && $-\frac2d$ & a.s. && $-\frac2d$ & a.s. \\
ES + $r$ & $- \frac12$ & high prob. &&  $- \frac12$ & high prob.  && $- \frac12$ & high prob. \\
MES+ $r$ & $- \frac12$  & high prob. &&  $\mathbf{ - \frac12 - \frac2d }$ & \textbf{high prob.}  && $\mathbf{ - \frac12 - \frac2d } $ & \textbf{high prob.} \\ 
\bf{Stochastic Gradient}\\
Shamir & $\mathbf{-1}$& \textbf{expect.} && $\mathbf{0}$ & \textbf{expect.} && $\mathbf{-1}$ & \textbf{expect.} \\
Shamir for ASR  & $\mathbf{ -1 }$& \textbf{expect. }&& $\mathbf{-1}$ & \textbf{expect.} && $\mathbf{-1}$  & \textbf{expect.}\\
Fabian & $\mathbf{-(1-e)}$  & \textbf{a.s.}  && $\mathbf{-e'}$  & \textbf{a.s.} && $\mathbf{-(1-e)}$& \textbf{a.s.}\\
Fabian for ASR & $\mathbf{-(1-e)}$ & \textbf{a.s.} && $\mathbf{-(1-e)}$  & \textbf{a.s.} &&  $\mathbf{-(1-e)}$ & \textbf{a.s.} \\
\bottomrule
\end{tabular}
\caption{\label{maintable}Convergence rates for the regrets analysed on this paper. The ``convergence'' column refers to the convergence rate and the ``type'' column specifies the type of convergence: with high probability, in expectation, almost surely. The results in bold are proved and the others are conjectures, all of them presented in this paper. 
}
\end{sidewaystable}
\vfill\break
\section{Experiments}\label{sec:experiments}
We present experimental results for part of the algorithms theoretically analysed \footnote{In addition, the experimental results we include the algorithm $UHCMAES$, as another example of an $ES$. For more information, see ~\cite{hansen2009tec}.}.  We will analyse the convergence rate of this algorithms in terms of \emph{slope} (see section \ref{sec:criteria} for the definition). As in the theoretical part, 
the function to optimize is the noisy sphere: $f(x)=\|x-x^*\|^2+\vartheta\NN$ where $\vartheta=0.3$ and $\NN$ is a standard gaussian distribution\footnote{ The choice $\vartheta=0.3$ is made only to illustrate in the experiments the effect of the regret choice in a reasonable time bugdet. The noise is weaker than in the case of a standard gaussian and the algorithms can deal with it faster. The optimum $x^*$ for the experiments of each algorithm is different, which does not affect the result since the regret compares the function value on the search/recommended points and on the optimum.}. The dimension of the problem is $2$. The results in Fig. \ref{zola} correspond to the mean of 10 runs for each of the algorithms.
\begin{table}[h]
\scriptsize
	\begin{center}
		\begin{tabular}{rl}
		\toprule
	 		Algorithms & Set of parameters \\
	 	\midrule
		UHCMAES ~\cite{hansen2009tec} &  $x_{\text{initial}}=1$,  $\sigma_{\text{in}}=1$ \\
		Shamir &$\epsilon=0.3$,  $\lambda=0.1$, $B=3$\\
		$(1+1)$-ES & \\
		$(1+1)$-ES resamp & resamp$=2^n$  \\
		Fabian & s=4 $\alpha=1$,  $\gamma=0.01$	 \\
		\bottomrule
 		\end{tabular}
	\end{center}
\label{table:exps} 
\end{table}
\begin{figure*}[ht]
\begin{minipage}{\textwidth}
	\centering
	\subfloat[Simple Regret]{\label{fig:SR}\includegraphics[width=0.45\textwidth]{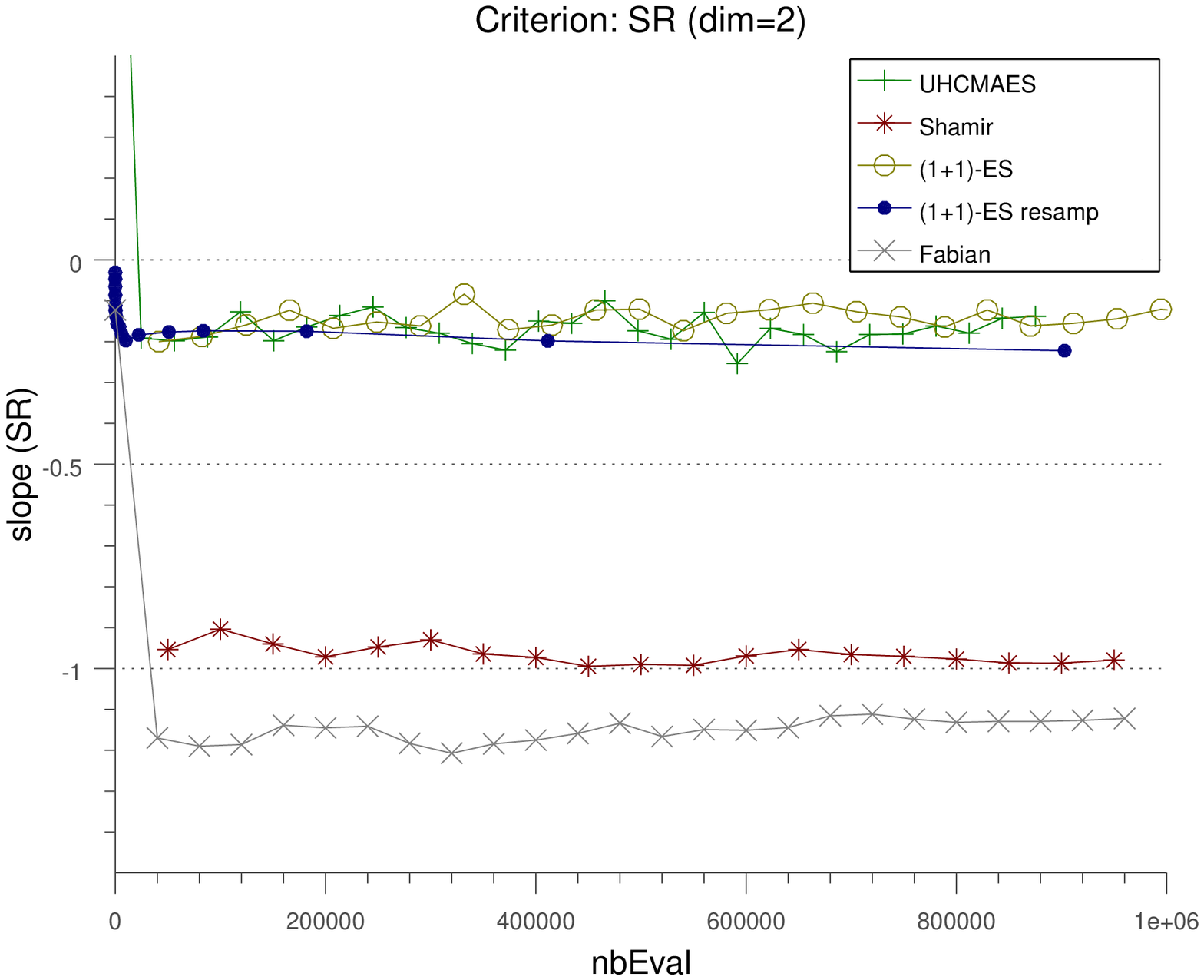}}
	\subfloat[Approximate Simple Regret]{\label{fig:ASR}\includegraphics[width=0.45\textwidth]{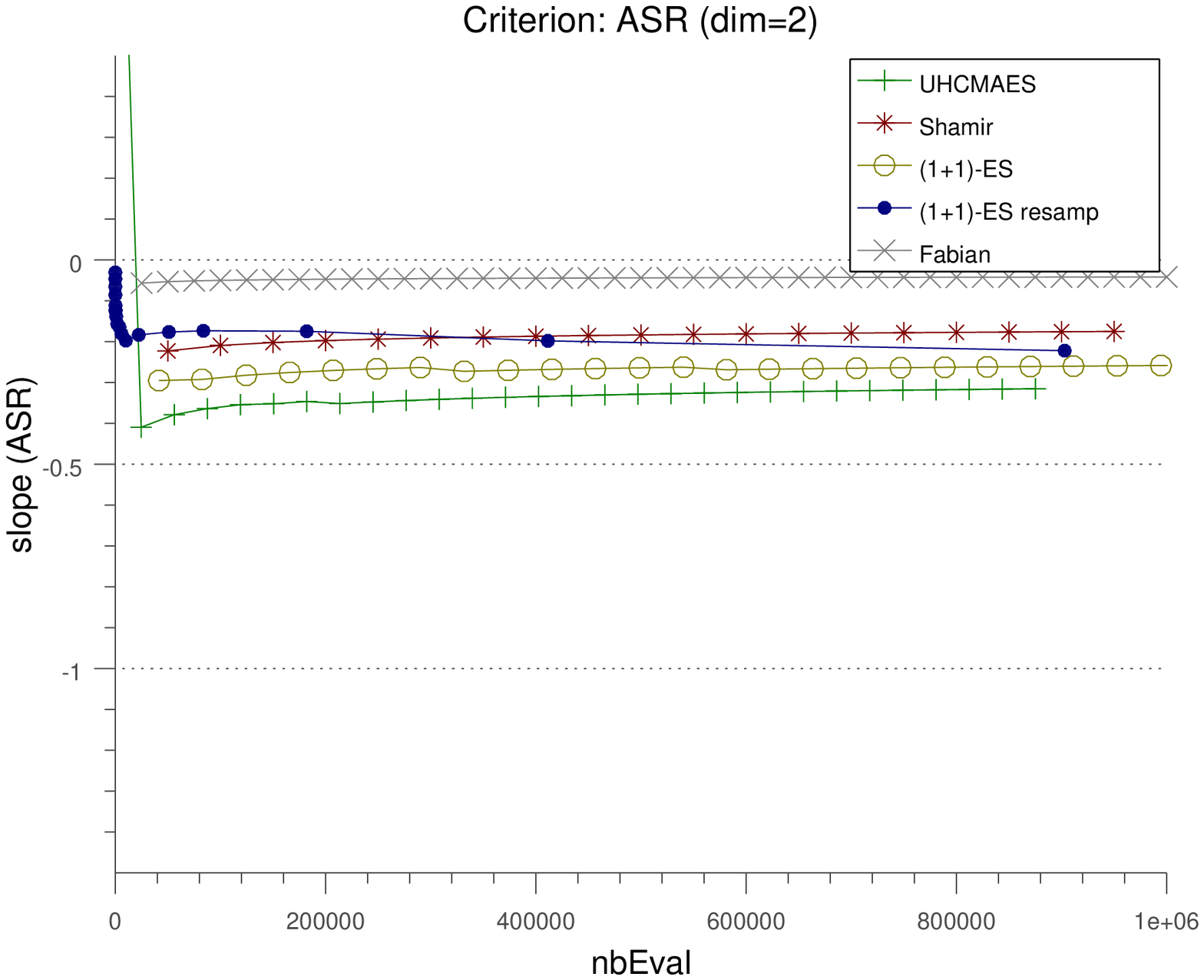}}
	\caption{\label{zola} Figure \ref{fig:SR} presents the Slope of Simple Regret for each algorithm on the first $(1\cdot10^6)$ function evaluations. Stochastic Gradient algorithms reach $s(SR)=-1$ while the evolutionary algorithms present $s(SR)=-0.2$. Figure \ref{fig:ASR} presents the Slope of Approximate Simple Regret. Observe that the performance of the algorithms is inverted with regards to the figure \ref{fig:SR} : now the Stochastic Gradient algorithms have the worse performance.}
\end{minipage}
\end{figure*}

The results in figure (\ref{fig:SR}) show the comparison between algorithms with regards to the $SR$ criterion. The budget is limited to $(1\cdot10^6)$ function evaluations. We can see clearly the difference between the algorithm that are gradient-based and the $ES$. The algorithms Fabian and Shamir achieve $s(SR)=-1$ whereas the $ES$ presented cannot do better than $s(SR)=-0.25$. The figure
 (\ref{fig:ASR}) shows that the use of $ASR$ changes completely the performance of the algorithms. In this case, the gradient-based algorithms are the ones with the worst performance. The results confirm the theoretical work (and the conjectures) presented in sections \ref{sec:EA} and \ref{sec:stoc}.


\section{Conclusion}\label{sec:conc}
In this paper we analyse the use of regrets to measure performance of algorithms in noisy optimization problems. We take into account the \emph{Simple Regret} ($SR$) and other two forms of regret used in practice to approximate the $SR$. We show that the convergence rates of the same algorithm over the same class of functions depend on the considered regret. Ultimately this leads to inconsistent results: algorithms are efficient for $SR$ and the opposite for an approximation of $SR$.
Table \ref{maintable} summarizes our results, detailing if they are proved or conjectured and what type of convergence is found.

{\bf{Approximations of $SR$.}} The $ASR$ appears to be a poor approximation of $SR$. Two situations are exposed in this paper.
First, algorithms with fast convergence for $SR$ can have a slow convergence for $ASR$ (Fabian and Shamir algorithm). However, this can be solved by modifying the algorithm in order to sample, sometimes, a recommended point. 
Second, an algorithm with slow convergence for $SR$ can have a fast convergence for $ASR$, and this is a serious issue. There is no simple ``patch'' to deal with this problem. Test beds using $ASR$ will underestimate algorithms which include random exploration. This is partially, but not totally, solved by $RSR$. We did not come up with a satisfactory regret definition, which would be consistent with $SR$ (at least showing similar convergence rates) and without drawbacks as those presented above.
However, the difference (between $SR$ on the one hand, and $ASR$/$RSR$ on the other hand) decreases with the dimension for most algorithms. 

{\bf{Simple Regret.}} 
$SR$ is clearly a natural way to measure the quality of an approximated optimum output by an algorithm. The drawback of $SR$ is that it is not necessarily non-increasing, which is an issue for the concept of ``stable first hitting time''. 
Note that a particularly interesting result is the one of Evolution Strategies correctly tuned in terms of its resampling scheme. It appears, from theoretical and experimental results, that it reaches half the speed of classical noisy optimization algorithms in terms of $SR$, on the log-log scale. This corresponds to a squared computation time. \footnote{This is not the case, however, for some ES with fast rates in noisy cases, such as mutate-large-inherit-small\cite{BeyerMutate}.}\\
{\bf{Further work.}} 
We have compared convergence rates through their \emph{slopes}, but in some cases we have slopes for almost sure convergence, in other cases slope with high probability, and in others slope in expectation. There is room for refinement of the results. As for the approximations of $SR$, there might be other definitions of regret that approximate it in a better way. Evidently, the use of an approximation makes sense in the case where there are technical issues that do not allow the direct computation of $SR$. 

\bibliographystyle{abbrv}
\bibliography{extracted}

\begin{thebibliography}{10}

\bibitem{abinvestigation}
D.~Arnold and H.-G. Beyer.
\newblock Investigation of the $(\mu,\lambda)$-es in the presence of noise.
\newblock In {\em Proc. of the IEEE Conference on Evolutionary Computation
  ({CEC} 2001)}, pages 332--339. IEEE, 2001.

\bibitem{abnoise}
D.~Arnold and H.-G. Beyer.
\newblock Local performance of the (1 + 1)-es in a noisy environment.
\newblock {\em Evolutionary Computation, IEEE Transactions on}, 6(1):30 --41,
  feb 2002.

\bibitem{stocopti5}
D.~V. Arnold and H.-G. Beyer.
\newblock A general noise model and its effects on evolution strategy
  performance.
\newblock {\em IEEE Transactions on Evolutionary Computation}, 10(4):380--391,
  2006.

\bibitem{esareslow}
S.~Astete-Morales, M.-L. Cauwet, and O.~Teytaud.
\newblock {Evolution Strategies with Additive Noise: A Convergence Rate Lower
  Bound}.
\newblock In {\em {Foundations of Genetic Algorithms}}, Foundations of Genetic
  Algorithms, page~9, Aberystwyth, United Kingdom, 2015.

\bibitem{bignoise2}
S.~Astete-Morales, J.~Liu, and O.~Teytaud.
\newblock log-log convergence for noisy optimization.
\newblock In {\em Proceedings of EA 2013}, LLNCS, page accepted. Springer,
  2013.

\bibitem{TCSAnne04-corr}
A.~Auger.
\newblock Convergence results for (1,$\lambda$)-{SA}-{ES} using the theory of
  $\varphi$-irreducible {M}arkov chains.
\newblock {\em Theoretical Computer Science}, 334(1-3):35--69, 2005.

\bibitem{nbbob1}
A.~Auger, D.~Brockhoff, and N.~Hansen.
\newblock Benchmarking the (1, 4)-cma-es with mirrored sampling and sequential
  selection on the noisy bbob-2010 testbed.
\newblock In {\em GECCO (Companion)}, pages 1625--1632, 2010.

\bibitem{nbbob2}
A.~Auger, D.~Brockhoff, and N.~Hansen.
\newblock Investigating the impact of sequential selection in the (1, 4)-cma-es
  on the noisy bbob-2010 testbed.
\newblock In {\em GECCO (Companion)}, pages 1611--1616, 2010.

\bibitem{nbbob4}
A.~Auger, D.~Brockhoff, and N.~Hansen.
\newblock Mirrored variants of the (1, 2)-cma-es compared on the noisy
  bbob-2010 testbed.
\newblock In {\em GECCO (Companion)}, pages 1575--1582, 2010.

\bibitem{BeyerMutate}
H.-G. Beyer.
\newblock {Mutate Large, But Inherit Small! On the Analysis of Rescaled
  Mutations in $(\tilde{1}, \tilde{\lambda})$-ES with Noisy Fitness Data}.
\newblock In {\em Parallel Problem Solving from Nature,~5}, Heidelberg, 1998.
  Springer.
\newblock in print.

\bibitem{Beyer:bookES}
H.-G. Beyer.
\newblock {\em The Theory of Evolution Strategies}.
\newblock Natural Computing Series. Springer, Heideberg, 2001.

\bibitem{BubeckE09}
S.~Bubeck, R.~Munos, and G.~Stoltz.
\newblock Pure exploration in multi-armed bandits problems.
\newblock In {\em ALT}, pages 23--37, 2009.

\bibitem{bignoise3}
M.-L. Cauwet.
\newblock {Noisy Optimization: Convergence with a Fixed Number of Resamplings}.
\newblock In {\em {EvoStar}}, Granada, Espagne, Apr. 2014.

\bibitem{chen1988}
H.~Chen.
\newblock {Lower rate of convergence for locating the maximum of a function}.
\newblock {\em {Annals of statistics}}, 16:1330--1334, Sept. 1988.

\bibitem{fogaasr}
D.-C. Dang and P.~K. Lehre.
\newblock Efficient optimisation of noisy fitness functions with
  population-based evolutionary algorithms.
\newblock In {\em Foundations of Genetic Algorithms 13}, Lecture Notes in
  Computer Science, page Accepted. Springer-Verlag, Berlin Heidelberg, 2015.

\bibitem{deheuvels}
P.~Deheuvels.
\newblock Strong bounds for multidimensional spacings.
\newblock {\em Z. fur Wahrsch. Verw. Geb.}, 64:411--424, 1983.

\bibitem{dupac}
V.~Dupa\v{c}.
\newblock {O} {K}iefer-{W}olfowitzov\v{e} {a}proxima\v{c}n\'{i} {M}ethod\v{e}.
\newblock {\em \v{C}asopis pro p\v{e}stov\'{a}n\'{i} matematiky},
  082(1):47--75, 1957.

\bibitem{fabian}
V.~Fabian.
\newblock {Stochastic Approximation of Minima with Improved Asymptotic Speed}.
\newblock {\em {Annals of Mathematical statistics}}, 38:191--200, 1967.

\bibitem{nbbob3}
S.~Finck and H.-G. Beyer.
\newblock Benchmarking spsa on bbob-2010 noisy function testbed.
\newblock In {\em GECCO (Companion)}, pages 1665--1672, 2010.

\bibitem{beynoise}
S.~Finck, H.-G. Beyer, and A.~Melkozerov.
\newblock Noisy optimization: a theoretical strategy comparison of es, egs,
  spsa {\&} if on the noisy sphere.
\newblock In N.~Krasnogor and P.~L. Lanzi, editors, {\em GECCO}, pages
  813--820. ACM, 2011.

\bibitem{hansen2009tec}
N.~Hansen, S.~Niederberger, L.~Guzzella, and P.~Koumoutsakos.
\newblock A method for handling uncertainty in evolutionary optimization with
  an application to feedback control of combustion.
\newblock {\em IEEE Transactions on Evolutionary Computation}, 13(1):180--197,
  2009.

\bibitem{nbbob9}
N.~Hansen and R.~Ros.
\newblock Benchmarking a weighted negative covariance matrix update on the
  bbob-2010 noisy testbed.
\newblock In {\em GECCO (Companion)}, pages 1681--1688, 2010.

\bibitem{nbbob7}
A.~LaTorre, S.~Muelas, and J.~M. Peña.
\newblock Benchmarking a mos-based algorithm on the bbob-2010 noisy function
  testbed.
\newblock In {\em GECCO (Companion)}, pages 1725--1730, 2010.

\bibitem{rastrigin1964convergence}
L.~Rastrigin.
\newblock Convergence of random search method in extremal control of many
  parameter system.
\newblock {\em Automation and Remote Control}, 24(11):1337, 1964.

\bibitem{Rechenberg73}
I.~Rechenberg.
\newblock {\em Evolutionsstrategie}.
\newblock Friedrich Frommann Verlag (G{\"u}nther Holzboog KG), Stuttgart, 1973.

\bibitem{nbbob5}
R.~Ros.
\newblock Black-box optimization benchmarking the ipop-cma-es on the noisy
  testbed: comparison to the bipop-cma-es.
\newblock In {\em GECCO (Companion)}, pages 1511--1518, 2010.

\bibitem{nbbob6}
R.~Ros.
\newblock Comparison of newuoa with different numbers of interpolation points
  on the bbob noisy testbed.
\newblock In {\em GECCO (Companion)}, pages 1495--1502, 2010.

\bibitem{Schw74b}
H.-P. Schwefel.
\newblock {Adaptive Mechanismen in der biologischen Evolution und ihr Einfluss
  auf die Evolutionsgeschwindigkeit}.
\newblock {Interner Bericht der Arbeitsgruppe Bionik und Evolutionstechnik am
  Institut f{\"u}r Mess- und Re\-ge\-lungs\-tech\-nik} Re 215/3, Technische
  Universit{\"a}t Berlin, Juli 1974.

\bibitem{shamir}
O.~Shamir.
\newblock On the complexity of bandit and derivative-free stochastic convex
  optimization.
\newblock In {\em {COLT} 2013 - The 26th Annual Conference on Learning Theory,
  June 12-14, 2013, Princeton University, NJ, {USA}}, pages 3--24, 2013.

\bibitem{spall00adaptive}
J.~Spall.
\newblock Adaptive stochastic approximation by the simultaneous perturbation
  method.
\newblock {\em Automatic Control, IEEE Transactions on}, 45(10):1839--1853, Oct
  2000.

\bibitem{nbbob8}
T.-D. Tran and G.-G. Jin.
\newblock Benchmarking real-coded genetic algorithm on noisy black-box
  optimization testbed.
\newblock In {\em GECCO (Companion)}, pages 1739--1744, 2010.

\end{thebibliography}




\section*{Supplementary material}

\begin{customthm}{1}
With the framework above, for all $\beta>0$, the expected simple regret $\E(SR_n)$ is not $O(n^{-\beta})$.
\end{customthm}

\begin{lemma}[Logarithmic bounds on the quantile of the standard Gaussian variable]\label{lemma:quantile}
Let $Q(q)$ be $q$ quantile of the standard centered Gaussian, i.e. $\forall q\in (0,1), {P(\NN \leq Q(q) ) =q}$. Then $\forall \kappa\geq 1$, $\forall q\in (0,1)$, $$1-\sqrt{-\frac{2}{\kappa}\log(c(\kappa)q)}\leq Q(q)\leq 1-{\sqrt{-2\log(2q)}},$$
where $c(\kappa)$ is a constant depending only on $\kappa$. In the following, we will denote by $X(q)$ (resp. $Y(q)$) the lower (resp.upper) bound on $Q$: $X(q)=1-\sqrt{-\frac{2}{\kappa}\log(c(\kappa)q)}$ and $Y(q)=1-{\sqrt{-2\log(2q)}}$.    
\end{lemma}

\begin{proof}
See \url{http://arxiv.org/pdf/1202.6483v2.pdf}.
\end{proof}

\begin{definition}\label{def:eps}
We recall that we consider $n$ i.i.d search points  $x_1,\dots,x_n$. Let $o$ and $\Omega$ be the standard Landau notations. Let $\e:\mathbb{N}\setminus\{0\}\mapsto\mathbb{R}$ and $C:\mathbb{N}\setminus\{0\}\mapsto\mathbb{R}$ be two functions satisfying:
\begin{itemize}
\item $\forall n\in\mathbb{N}\setminus\{0\}$, $\epsilon(n)>0$ and $C(n)>0$,
\item $\forall n\in\mathbb{N}\setminus\{0\}$, $\epsilon(n) =o( C(n) )$,
\item $\forall n\in\mathbb{N}\setminus\{0\}$, $\epsilon(n)=\Omega(1/n)$ and $C(n)=\Omega(1/n)$,
\item $\forall n\in\mathbb{N}\setminus\{0\}$, $\epsilon(n)=o(1)$ and $C(n)=o(1)$.
\end{itemize}
\end{definition}

\begin{definition}
With the previous definition of $\e$ and $C$, consider the set $G$ defined by
$$G:=\{i\in \{1,\dots,n\}; \|x_i-x^*\|\leq \sqrt{\epsilon(n)}\}.$$
$G$ is the set of ``good'' search points, with simple regret better than $\epsilon(n)$. We denote by $N_G(n)$ the cardinality of $G$.
\end{definition}

\begin{definition}
Similarly, consider the set $B$ defined by
$$B:=\{i\in\{1,\dots,n\}; \sqrt{\e(n)} < \|x_i-x^*\|\leq\sqrt{C(n)}\}.$$
$B$ is the set of ``bad'' search points, with simple regret bigger than $\epsilon(n)$, but still not that bad, since the simple regret does not exceed $C(n)$. We denote by $N_B(n)$ the cardinality of $B$.
\end{definition}


      
\begin{lemma}[Linear numbers of good and bad points]\label{lemmaCheb}
There exist a constant $K_d>0$ such that, with probability at least $1/2$,
\begin{equation}
N_G(n) < 2K_{d}n\sqrt{\epsilon(n)}^d\label{calvin}
\end{equation}  
and 
\begin{equation}
N_B(n) \geq 2 K_d n \sqrt{C(n)}^d.\label{hobbes}
\end{equation}
\end{lemma}

\begin{proof} 
Proof of Eq.~\ref{calvin}. Consider a search point $x_i$. The search points are drawn uniformly at random following the uniform distribution in $[0,1]^{d}$, then the probability $p$ that $x_i\in G$ is $p=\frac{Vol(B_d(x^*,\sqrt{\e(n)}))}{Vol([0,1]^{d})}= K_{d}\sqrt{\e(n)}^{d}$, where $Vol$ stands for `volume' and $K_{d}$ is a constant depending on $d$ only. Therefore the number $N_G(n)$ of good points is the sum of $n$ Bernoulli random variables with parameter $K_{d}\sqrt{\epsilon(n)}^d$. The expectation is then $K_{d}n\sqrt{\epsilon(n)}^d$.  
By Markov inequality, 
\begin{equation*}
\mathbb{P}(N_{G}(n)\geq 2K_{d}n\sqrt{\epsilon(n)}^d)\leq \frac{K_{d}n\sqrt{\epsilon(n)}^d}{2K_{d}n\sqrt{\epsilon(n)}^d}=\frac{1}{2}. 
\end{equation*}

Similarly, $N_B(n)$ is a binomial random variable of parameters $n$ and $p=K_d (\sqrt{C(n)}^d-\sqrt{\e(n)}^d)$. Then by Chebyshev's inequality,
$\mathbb{P}(|N_B(n)-K_d n (\sqrt{C(n)}^d-\sqrt{\e(n)}^d)|\leq \alpha)\geq 1/2$ by taking $\alpha=\sqrt{2K_dn(\sqrt{C(n)}^d-\sqrt{\e(n)}^d)(1-\sqrt{C(n)}^d+\sqrt{\e(n)}^d)}$. Hence with probability at least $1/2$, $N_B(n)\geq K_d n (\sqrt{C(n)}^d-\sqrt{\e(n)}^d)+\alpha\geq 2 K_d n \sqrt{C(n)}^d$ since $\e(n)=o(C(n))$.
\end{proof}

We recall that $\forall i\in\{1,\dots,n\}$, $y_i$ is the fitness value of search point $x_i$: $y_i=\|x_i - x^*\|^2 + \NN_i$, where $\NN_i$ is the realisation of a standard centered gaussian variable.

The following property gives a lower bound on the fitness values of the `good' points, and an upper bound on the fitness values of the `bad' points.

\begin{proposition}\label{zeprop}
With $X$ and $Y$ as defined in Lemma~\ref{lemma:quantile}, and $C$ as defined in Definition~\ref{def:eps}, there exists some $c\in(0,1)$ such that, with probability at least $c$,
\begin{equation}\label{eq:goodguyslowerbound}
\inf_{i \in G} y_i \geq X(1/N_G(n)),
\end{equation}
and
\begin{equation}\label{eq:badguysupperbound}
\inf_{i\in B} y_i \leq C(n)+Y(1/N_B(n)).
\end{equation}
\end{proposition}

\begin{proof}

Consider some Gaussian random variables independently identically distributed $\NN_1,\dots,\NN_N$. $\forall\ i\in \{1,\dots,N\}$, using notation of Lemma~\ref{lemma:quantile}, $\mathbb{P}(\NN_i\leq Q(1/N))=1/N$, then $\mathbb{P}(\inf_{1\leq i\leq N} \NN_i\leq Q(1/N))=1-\mathbb{P}(\inf_{1\leq i\leq N} \NN_i\geq Q(1/N)=1-(1-1/N)^N$. The study of the function $x\mapsto 1-(1-1/x)^x$ then shows that $\mathbb{P}(\inf_{1\leq i\leq N} \NN_i\leq Q(1/N))\in[1-\exp(-1),3/4]$ (as soon as $N\geq 2$). 

Proof of Eq.~\ref{eq:goodguyslowerbound}.
\begin{align*}
\mathbb{P}(\inf_{i \in G} y_i \geq X(1/N_G(n)))&=\mathbb{P}(\inf_{i \in G} \|x_i-x^*\|^2\\
&\ \ \ \ \ \ \ \ \ +\NN_i \geq X(1/N_G(n)))\\
&\geq \mathbb{P}(\inf_{i \in G} \|x_i-x^*\|^2\\
&\ \ \ \ \ \ \ \ \ \ +\NN_i \geq \e(n) + X(1/N_G(n)))\\
&\geq \mathbb{P}(\inf_{i \in G}\NN_i \geq X(1/N_G(n)))\\
&\geq 1-\mathbb{P}(\inf_{i \in G}\NN_i \leq X(1/N_G(n)))\\
&\geq 1-\mathbb{P}(\inf_{i \in G}\NN_i \leq Q(1/N_G(n)))\\
&\geq 1/4.
\end{align*}

Proof of Eq.~\ref{eq:badguysupperbound}.

\begin{align*}
\mathbb{P}(\inf_{i \in B} y_i \leq C(n)+Y(1/N_B(n)))&=\mathbb{P}(\inf_{i \in G} \|x_i-x^*\|^2+\NN_i\\
&\ \ \ \ \ \leq C(n)+Y(1/N_B(n)))\\
&=\mathbb{P}(\inf_{i\in B} \NN_i \leq Y(1/N_B(n)))\\
&\geq \mathbb{P}(\inf_{i\in B} \NN_i \leq Q(1/N_B(n)))\\
&\geq  1-\exp(-1) 
\end{align*}   
Hence, with probability at least $1/4$, Eqs.~\ref{eq:goodguyslowerbound}
and \ref{eq:badguysupperbound} hold.
\end{proof}


{\bf{Proof of Theorem \ref{rswn}:}}
\begin{proof}

Let us assume that the expected simple regret has a slope $-\beta_0<-\beta$, for some $0<\beta<1$: $\E(SR_n)=O(n^{-\beta_0})$. 

We define $\epsilon(n)=n^{-\beta}$. For some $0<k<1$, $\alpha=\beta/k$ and $C(n)=n^{-\alpha}$. $\e$ and $C$ satisfy Definition~\ref{def:eps}. 

Lemma \ref{lemmaCheb} and Proposition \ref{zeprop} implies that there is a $0<c<1$ such that with probability $c$, for $n$ sufficiently large,
\begin{itemize}
	\item there are much more good points than bad points, i.e.
\begin{equation}
	N_G(n)=o(N_B(n))\label{proutproutprout}
\end{equation}
 	thanks to Eq. \ref{calvin} and \ref{hobbes}.
	\item all good points have noisy fitness at least $X(1/N_G(n))=1-\sqrt{-\frac{2}{\kappa}\log(c(\kappa)/N_G(n))}$;
	\item at least one bad point has fitness at most $C(n)+Y(1/N_B(n))=n^{-\alpha}+1-{\sqrt{-2\log(2/N_B(n))}})$; 
	\item therefore (by the two points above, using Eq. \ref{proutproutprout} and the fact that $n$ is big so that $n^{-\alpha}$ is negligeable), at least one bad point has a better noisy fitness than all the good points, and therefore is selected in Lines 7-9 of Alg. 1.
\end{itemize}

This implies that 
\begin{align*}
\E(SR_n)&=\E(\|x_n-x^*\|^2)=\E(\|x_n-x^*\|^2| x_n \in B)\mathbb{P}(x_n \in B)\\
& +\E(\|x_n-x^*\|^2| x_n \in G)\mathbb{P}(x_n \in G)\\
&\geq \E(\|x_n-x^*\|^2| x_n \in B)\mathbb{P}(x_n \in B)\\
&\geq \e(n)\times c\\
&\geq c n^{-\beta}.
\end{align*}
We have a contradicion, hence $\beta_0 >0$.
\end{proof}

\begin{customthm}{2}
Consider $0<\delta<1$.
Consider an objective function with expectation $F(x)=\|x\|^2$, where $x\in \R^d$.
Consider a $MES+R$ as described previously (resampling number exponential in the number of iterations). Assume that:
\begin{enumerate}[label=(\roman*)]
\item $\sigma_n$ and $\|\tilde{x}_n\|$ have the same order of magnitude: 
\begin{equation}
\|\tilde{x}_n\|=\Theta(\sigma_n).
\end{equation}
\item $\log-\log$ convergence occurs for the $SR$: 
\begin{equation}
\frac{\log(\|\tilde{x}_n\|)}{\log(n)}\underset{n\rightarrow+\infty}{\longrightarrow }-\frac12 ~\mbox{ with probability $1-\delta$,}
\end{equation}
\end{enumerate}

Then, with probability at least  $1-\delta$, $s(ASR)=-1/2-2/d$.
\end{customthm}

\begin{proof}
We have $r_n = R \zeta^n$.
In this proof we will index the recommendation and search points  by the number of \emph{iterations} instead of the number of evaluations. For $ES+r$, the recommendation point of the iteration $n$ is  the corresponding center of the offspring distribution $\tilde x_n$. The step-size $\sigma_n$ corresponds to the standard deviation of the offspring distribution. The search points of iteration $n$ are the $\lambda$ offsprings produced on the iteration, denoted $\{ x_n^{(i)}: i=1,\ldots,\lambda \}$.

We define $e_n$ the number of evaluations until the iteration $n$. From Algorithm 2 
we have $e_n=\lambda \sum_{i=1}^{n} r_i$ for an ES with resampling $r$. 
By hypothesis (Eq.~\ref{loglog}) we know the convergence  of the sequence $S_n$ defined as the logarithm of the recommendation points, indexed by the number evaluations, divided by the logarithm the number of evaluations. Therefore, the sequence $\mathcal{S}_n = \log (||x_n||) /\log (e_n)$ is a subsequence of $S_n$, hence convergent to the same limit, with the same probability. The relation on Eq.~\ref{xsigma} also remains the same after the index modification. From these facts we can inmediatly conclude that $\exists \quad n_0$ such that
\begin{equation}\label{ensigma}
\sigma_{n}=\Theta(e_n^{-1/4}) \mbox{ for $n\geq n_0.$}
\end{equation}
For the $MES+r$ algorithm, the $ASR_n$ is defined as:
\begin{equation}
{ASR_n =\min_{m\leq n} \min_{1\leq i \leq \lambda+r_m}\|{x'}_{m}^{(i)}\|^{2}}
\end{equation}
where $\{ {x'}^{(i)}_n : i= 1,\ldots,\lambda+r_m \}$ considers all the search points at iteration $n$. Both the ``real'' and the ``fake'' ones. We will find a bound for $ASR_n$, which will lead us to the convergence rate of the $ASR$ for the $MES+r$ algorithm.
 
Let $p_x$ be the probability density at point $x$ of the offsprings in iteration $n$. Therefore it is the probability density of a Gaussian centered at $\tilde{x}_n$ and with variance $\sigma_n^2$. At the origin:
\begin{align}
p_0&=\frac{1}{(2\pi)^{d/2}\sigma_{n}^{d}}\exp\left\{-\frac12 \frac{(-\tilde{x}_n)^{T}(-\tilde{x}_n)}{\sigma_{n}^2}\right\}\nonumber\\
     &=\Theta(\sigma_{n}^{-d})~\mbox{by Eq. \ref{xsigma}}\nonumber\\
     &=\Theta(e_n^{d/4})~\mbox{by Eq. \ref{ensigma}}\label{densitybound}
\end{align}

Now, at iteration $n$, the probability to have at least one offspring with norm less than $\epsilon>0$ is  
\begin{eqnarray*}
\PP(\exists i : \|{x'}_{n}^{(i)}\|\leq \epsilon)&\leq (\lambda + r_n )\PP(\|{x'}_{n}^{(i)}\| \leq \epsilon)\\
				      			    &\leq (\lambda +r_n) \int_{\|x\| \leq \epsilon}dp_x
\end{eqnarray*}
By Eq.~\ref{densitybound},

$$\PP(\exists i | \|{x'}_{n}^{(i)}\|\leq \epsilon)=\Theta(r_n\cdot e_n^{d/4} \epsilon^{d})=\Theta(1)$$
 if $\epsilon = \Theta(  e_n^{-1/4} \cdot r_n^{-1/d} )$.
 Therefore we obtain:
\begin{align*}
ASR_n &\leq \min_{1\leq i \leq \lambda+r_n} \|{x'}_{n}^{(i)}\|^{2}\\
			&\leq\epsilon^{2}\\
			&=O(e_n^{-1/2} \cdot r_n^{-2/d}) \\
			&=O(e_n^{-1/2-2/d})
\end{align*}
Since $e_n=(1+\lambda)\sum_{i=1}^{n} r_{i}=(1+\lambda)\cdot R\sum_{i=1}^{n}\zeta^{i}=\Theta(r_n)$
Hence we have the result, $s(ASR)\leq -1/2-2/d$. 
\end{proof}

\begin{customthm}{5}[$ASR$ of Fabian algorithm]
Let $F$ be a $\lambda$-convex and $\mu$-smooth function corrupted by an additive noise with upper bounded density and with optimum randomly drawn according to a distribution with upper bounded density. Then, a.s.,
\begin{equation*}
s(ASR)= -\min(2\beta,2\gamma).
\end{equation*}
\end{customthm}
\begin{proof}
The upper bounded density is used (Theorem 3 section 4.1)
) for ensuring that no recommendation is never exactly equal to the optimum.

Using the algorithm notations, $$ASR_n= {\min_{n,i,j} F(x^{(i,j)+}(n)) - F(x^*)}$$ (or ${\min_{n,i,j} F(x^{(i,j)-}(n)) - F(x^*)}$, which does not affect the proof), where $x^{(i,j)+}(n)$ is the $n^{th}$ search point. Noting that since there is $s\times d$ evaluations per iterations, $\tilde{x}$ is updated only every $s\times d$ evaluations, so we have $\tilde{x}_n=\tilde{x}_{n+1}=\dots=\tilde{x}_{n+s\times d -1}$, therefore $x^{(i,j)+}(n)=\tilde{x}_{\left\lfloor n/s\times d\right\rfloor} +c_n u_j e_i$. 
By the convexity of $F$, the fact that the gradient of $F$ in $x^*$ is $0$, 
\begin{align*}
F(x^{(i,j)+}(n)) -& F(x^*)\geq \frac{\lambda}{2}\| (\tilde{x}_{\left\lfloor n/s\times d\right\rfloor} -x^*)+ (c_n u_j e_i) \|^2
\end{align*}
Similarly, by using the $\mu$-smoothness of $F$,
\begin{align*}
F(x^{(i,j)+}) - F(x^*) \leq\frac{\mu}{2}\| (\tilde{x}_{\left\lfloor n/s\times d\right\rfloor} -x^*)+(c_n u_j e_i)\|^2
\end{align*}
Then $$F(x^{(i,j)+}(n)) - F(x^*)=\Theta(\| (\tilde{x}_{\left\lfloor n/s\times d\right\rfloor} -x^*)+ (c_n u_j e_i)\|^2)$$
If $\beta >\gamma$, then the main term is the last one and we get a rate $-2\gamma$.
If $\beta \leq\gamma$, then the main term is the first one and we get a rate $-2\beta$.

%
%
\end{proof}

\end{document}